\documentclass[12pt,onecolumn,draftclsnofoot]{IEEEtran}
\usepackage{times}
\usepackage[utf8]{inputenc}
\usepackage{color}
\usepackage{amsmath}
\usepackage{times}
\usepackage{graphicx}
\usepackage{color}
\usepackage{multirow}
\usepackage{blindtext}
\usepackage{amsfonts}
\usepackage{amsmath}
\usepackage{times}
\usepackage{fancyhdr}
\usepackage{amssymb}
\usepackage{verbatim}
\usepackage{xspace}
\usepackage{cite}
\usepackage{comment}
\usepackage{lipsum}
\usepackage{bm}
\usepackage{hyperref}
\usepackage{subfigure}

\usepackage[normalem]{ulem}

\usepackage{algorithm}
\usepackage[]{algpseudocode}

\usepackage{amsthm}

\newcommand{\rd}{{\mathrm d}}

\newcommand{\re}{{\mathrm e}}

\newcommand{\calH}{{\cal H}}
\newcommand{\calF}{{\cal F}}

\newcommand{\calU}{{\cal U}}

\newcommand{\vcalN}{{\mbox{\boldmath$\cal{N}$}}}

\newcommand{\Pb}{\mathbb{P}}

\newcommand{\Rb}{\mathbb{R}}
\newcommand{\Eb}{\mathbb{E}}

\newtheorem{theorem}{Theorem}[section]
\newtheorem{corollary}{Corollary}[theorem]
\newtheorem{lemma}[theorem]{Lemma}

\newtheorem{remark}{Remark}
\newtheorem{definition}{Definition}[section]

\begin{document}	
	\title{Schr\"odinger Approach to Optimal Control of Large-Size Populations}
	\author{Kaivalya Bakshi, David D. Fan and Evangelos A. Theodorou
	\thanks{Kaivalya Bakshi is a researcher in the Controls and Optimization group, GE Research, Niskayuna, NY 12309 USA  (email: kaivalya.bakshi@ge.com)}
	\thanks{David D. Fan is a PhD candidate in the Department of Electrical and Computer Engineering, Georgia Institute of Technology, Atlanta,
		GA 30332 USA}
	\thanks{Evangelos A. Theodorou is in the Department of Aerospace Engineering, Georgia Institute of Technology, Atlanta,
		GA 30332 USA}
			}
	
	\date{\today}
	
	\maketitle

\begin{abstract}
Large-size populations consisting of a continuum of identical and non-cooperative agents with stochastic dynamics are useful in modeling various biological and engineered systems. This paper addresses the stochastic control problem of designing optimal state-feedback controllers which guarantee the closed-loop stability of the stationary density of such agents with nonlinear Langevin dynamics, under the action of their individual steady state controls. \\
\indent We represent the corresponding coupled forward-backward PDEs as decoupled Schr\"odinger equations, by applying two variable transforms. Spectral properties of the linear Schr\"odinger operator which underlie the stability analysis are used to obtain explicit control design constraints. Our interpretation of the Schr\"odinger potential as the cost function of a closely related optimal control problem motivates a quadrature based algorithm to compute the finite-time optimal control \footnote{Code publicly available at \url{https://github.com/ddfan/pi_quadrature}}.
\end{abstract}

\section{Introduction}
\label{sec:intro}

Dynamics and control of multi-agent populations consisting of a large number of identical and non-cooperative agents are of interest in various applications including robotic swarms, macro-economics, traffic and neuroscience. Prior works on optimal open-loop or closed-loop ensemble (broadcast) control consider several copies of a particular deterministic \cite{Brocket2012} or stochastic (\cite{Brockket2000}, \cite{Milutinovic2011}, \cite{Bakshi_CDC2016}) system and have applications in quantum control \cite{Ruths2011} and neuroscience \cite{Ruths2013}. A standard idea in engineering, economics and biology to model decision making in large-size populations of agents with limited information is regulation using local feedback information. Optimal feedback control applications to {large-size} populations of \textit{small} robots with individual state-feedback controllers have been proposed for inspection of industrial machinery \cite{Correll2009}, centralized control of hybrid automata \cite{Milutinovic2013} and decentralized control of robotic bee swarms for pollinating crops \cite{LiBerman2017}. In large-size multi-agent systems wherein the dimensions of individual agents are small compared to their region of operation, it can be assumed that the agents do not locally interact with each other. Such problems with stochastic agents can therefore be treated within the framework of standard stochastic optimal control.

In this paper we consider the finite and infinite-time nonlinear optimal control problem (OCP) of a density of identical and non-cooperative agents which have individual state-feedback controllers. The distribution of agents (referred to as population in this paper) is represented by the density function of the state of a representative agent. The standard stochastic OCP \cite{Fleming2006} is represented by two coupled forward-backward partial differential equations (PDEs) known as Hamilton-Jacobi-Bellman (HJB) and Fokker-Planck (FP), which govern the value function and agent density respectively. Note that although the FP equation depends on the value function, the HJB equation does not depend on the density. The important question we address in this paper is: how do we design the cost function of the OCP so that the steady state optimal controller stabilizes an initial (perturbed) density of agents to the corresponding steady state density? This question clearly relates to the stability analysis of the FP equation governing the population density.

Mean field games (MFGs) (\cite{Lasry2007}, \cite{Huang2007}) constituted by  game-theoretic OCPs use a similar HJB-FP PDE representation to model continuum systems in which agents \textit{interact} with each other through the dependence of their dynamics or cost function on the population density (\cite{kachroo2016inverse}, \cite{carmona2013mean}, \cite{couillet2012electrical}).
In the recent works (\cite{Bakshi2018TCNS}, \cite{grover2018mean}) by some of the authors and in (\cite{Gueant2009}, \cite{yinmehmeysha10}, \cite{Nourian2014}), local (linear) stability results were presented for certain classes of MFGs wherein agents obey nonlinear Langevin dynamics. 
The presented work is therefore differentiated from these earlier works since it is concerned with the stability analysis of the population density related to standard stochastic OCPs in which the agents do not interact with each other.

The first contribution of this paper is the application of variable transforms which enable the representation of the HJB and FP PDEs corresponding to a class of stochastic OCPs with nonlinear agent dynamics, as decoupled imaginary-time linear Schrödinger equations.
The fundamental connection between optimal control and quantum mechanics (\cite{Zhou2017},\cite{SchrodChen2016}) has been shown through a standard Cole-Hopf transform of the value function (\cite{Cole1950},\cite{Hopf1950}) to  represent the HJB PDE governing the value function as the Schr\"odinger equation. In the presented work, we transform both the value \textit{and} density functions and represent the HJB and FP equations as decoupled Schr\"odinger equations. The variable transform of the value function in this paper is distinct from prior works including the work in \cite{Ullmo2016} on MFGs of agents with simple integrator dynamics, which utilizes a density transform similar to the what is presented in this paper, since it reflects the nonlinear dynamics of agents with Langevin dynamics.

The second contribution of this work is to show that the  Schr\"odinger potential is the cost function of a closely related (fictitious) OCP subject to simple integrator dynamics. This interpretation is closely related (in the case of uncontrolled stochastic dynamics) to the well known work of Nelson \cite{Nelson1966} which shows the connection between the classical physical interpretation of the Brownian motion and quantum mechanics. 

The main contribution of the paper is the stability analysis of the optimally controlled density of agents for a class of stochastic OCPs, wherein agents obey nonlinear dynamics. We obtain explicit (analytical) \textit{control design constraints} required for stability. This work presents more general \textit{nonlinear stability} results in contrast to the prior related works  (\cite{Bakshi2018TCNS}, \cite{grover2018mean},\cite{Gueant2009}, \cite{yinmehmeysha10}, \cite{Nourian2014}) in the MFG case, which rely on linearization of the PDEs representing the OCP. This is made possible by applying two variable transforms in this paper, which enables representation of the HJB and FP equations of the considered OCP by decoupled linear Schr\"odinger equations. Our work is inspired by the work \cite{Ullmo2019} which analyzes the stability of equilibrium densities of MFGs for agents with simple integrator dynamics. Although our work is limited in contrast to \cite{Ullmo2019} since it does not address the more general MFGs, it is more general since it treats agents with nonlinear Langevin dynamics. Further, the approach and consequent stability analysis presented in \cite{Ullmo2019}, is based on the Gaussian approximation of the equilibrium density. Therefore, although the presented stability analysis in \cite{Ullmo2019} is  applicable to non-Gaussian initial densities which can be approximated by a juxtaposition of several Gaussian bumps, the case of arbitrary non-Gaussian initial densities is beyond it's scope. In contrast, our work treats arbitrary non-Gaussian initial densities.

The fourth contribution of our work is a sampling algorithm to solve the original nonlinear finite-time OCP using trajectories sampled from the linear (integrator) dynamics of the related fictitious OCP. Prior path-integral based sampling control algorithms (\cite{Kappen2006}, \cite{Theodorou2011IPI}) rely on simulating trajectories of nonlinear dynamics  to numerically approximate probability distributions, thus introducing high computational complexity and inaccuracy. Since our algorithm uses samples from linear dynamics, we use analytic knowledge of the sample distribution via a quadrature method to compute the control.

The paper is organized as follows. In section \ref{sec:LargeScale} we state the OCP considered, the corresponding PDE representation and stationary solutions to it. In section \ref{sec:HopfCole} we apply two variable transforms of the value and density functions, obtain the resulting representation of the PDEs as Schr\"odinger equations and show that the Schr\"odinger potential of the transformed PDEs is the cost function of a closely related OCP with integrator dynamics. In section \ref{sec:Stab} we present the stability analysis for the class of OCPs for agents with Langevin dynamics, explicit control design constraints and prove that the population density of optimally controlled agents stabilizes to the equilibrium density under the optimal control resulting from the solution of an OCP which obeys the constraints. In section \ref{sec:solver} we present a sampling algorithm to compute the finite-time nonlinear stochastic control using samples from linear dynamics. In section \ref{sec:conclusions} we state our conclusions and point to directions for future research.

\section{Control of Large-Size Populations}
\label{sec:LargeScale}

We first introduce some notation and then describe the control problem considered in this work. We denote vector inner products by $a \cdot b$, the induced Euclidean norm by $|a|$ and its square by $a^2 = |a|^2$. 
$\partial_t$ denotes partial derivative with respect to $t$ while $\nabla$, $\nabla \cdot$ and $\Delta$ denote the gradient, divergence and Laplacian operations respectively. $L^2(\Rb^d)$ denotes the class of square integrable functions of $\Rb^d$. The norm of a function $f$ and inner product of functions $f_1, f_2$ in this class is denoted by $||f||_{L^2(\Rb^d)}$ and $\big< f_1,f_2 \big>_{L^2(\Rb^d)}$ respectively.

\subsection{Control problem}
\label{sec:LargeScale_OCP}

Let $x_s$, $u(s)$ $\in \Rb^d$ denote the state and control inputs of a representative agent which obeys the controlled first order dynamics:
\begin{equation}
	\rd x_s = - \nabla \nu(x_s) \rd s + u(s) \rd s + \sigma \rd w_s \label{dyn}
\end{equation}
for every $s \geq 0$, driven by standard $\Rb^d$ Brownian motion, with noise intensity $0 < \sigma$ on the filtered probability space $\{ \Omega, \calF, \{\calF_t\}_{t \geq 0},\Pb \}$. These dynamics are the controlled version of a Langevin system in the overdamped case. The smooth function $\nu:\Rb^d \rightarrow \Rb$ is called the Langevin potential and $u \in \calU := \calU[0,T]$ where $\calU$ is the class of admissible controls \cite{YongBook_1958} containing functions $u:[0,T] \times \Rb^d \rightarrow \Rb^d$. Consider the following optimal control problem (OCP):
\begin{equation}
\underset{u \in \calU}{\min} J(u) := \underset{T \rightarrow +\infty}{\lim} \frac{1}{T} \Eb \left[ \int_{0}^{T} q(x_s) \rd s + \frac{R}{2}u^2 \; \rd s \right] \label{OCP}
\end{equation}
subject to \eqref{dyn} which we call problem \textbf{(P1)}, wherein the expectation is calculated on the probability density $p(s,x)$ of $x_s$ for all $s \geq 0$ and represents the distribution of the population of agents, with the initial density being $x_0 \sim p(0,x)$, $q: \Rb^d \rightarrow \Rb$ is a known deterministic function which has at most quadratic growth and is bounded from below and $R > 0$ is the control cost. We assume that $\nabla \nu(\cdot), q(\cdot)$ and functions in the class $\calU$ are measurable.

\subsection{PDE representation}

Standard application of \textit{dynamic programming} \cite{Fleming2006} implies that under certain regularity conditions \cite{yinmehmeysha10} which we assume to be true, problem \textbf{(P1)} is equivalent to the following HJB and FP equations governing the value and density functions respectively:
\begin{align}
q - c - \frac{(\nabla v^\infty)^2}{2R} - \nabla v^\infty \cdot \nabla \nu + \frac{\sigma^2}{2} \Delta v^\infty =& 0 \label{s1} \\
\nabla ((\nabla \nu + \frac{\nabla v^\infty}{R})p^\infty) + \frac{\sigma^2}{2} \Delta p^\infty =& 0 \label{s2}
\end{align}
with the constraint $\int p^\infty \rd x = 1$, where $c$ is the optimal cost. The optimal control is given by $u^\infty(x) = -\nabla v^\infty/R$. Under certain regularity conditions  \cite{yinmehmeysha10}, which we assume to be true,
the time-varying relative value \cite{Borkar2006} function and density corresponding to problem \textbf{(P1)} are governed by the following equations:
\begin{align}
-\partial_t v =& q -c - \frac{(\nabla v)^2}{2R} - \nabla v \cdot \nabla \nu + \frac{\sigma^2}{2} \Delta v \label{HJB} \\
\partial_t p =& \nabla \cdot ((\nabla \nu + \frac{\nabla v}{R})p) + \frac{\sigma^2}{2} \Delta p \label{FP}
\end{align}
with the constraint $\int p(t,x) \rd x = 1$ for all $t \geq 0$. The optimal control is given by $u^*(t, x) = -\nabla v/R$. In this work, we assume the additional conditions \cite{YongBook_1958} which are required to show that the HJB PDEs \eqref{s1} and \eqref{HJB} have unique solutions. Both steady state and time-varying HJB PDEs are semilinear.

\begin{remark}\label{finitetime}
	The finite-time OCP analogous to the infinite-time OCP \textbf{(P1)} given by:
	\begin{equation}
	\underset{u \in \calU}{\min} J(u) := \Eb \left[ \int_{0}^{T} q(x_s) \rd s + \frac{R}{2}u^2 \; \rd s \right]. \label{finitetimeOCP}
	\end{equation}
	subject to the dynamics \eqref{dyn} has the PDE representation given by equations \eqref{HJB}, \eqref{FP} with $c = 0$, initial density given by $p(0,x)$ and constraint $\int p(t,x) \rd x = 1$.
\end{remark}


\subsection{Stationary solution}

The FP equation governing the density of an overdamped Langevin system is called the Smoluchowski PDE. The FP PDE \eqref{s2}, can be interpreted as the Smoluchowski PDE for such a Langevin system with the restoring potential $\nu + v^\infty/R$. The analytical solution to this FP equation can be obtained as a Gibbs distribution using this interpretation, under certain conditions given below, on the solution pair $(v^\infty,p^\infty)$ of the equations (\ref{s1}, \ref{s2}) and the Langevin potential $\nu$. We denote $w(x) := \nu(x) + \frac{v^\infty(x)}{R}$. We state an assumption below which is required to obtain analytic solutions to the FP equation. Note that the stationary densities obtained using the following lemma may be non-Gaussian.
\begin{itemize}
	\item[\textbf{(A0)}] There exist ($v^\infty(x), p^\infty(x)) \in (C^2(\Rb^{d}))^2$ satisfying (\ref{s1},\ref{s2}) such that
	{$\lim\limits_{|x| \rightarrow +\infty} w(x) = +\infty$} and {$\exp\left({-\frac{2}{\sigma^2} w(x)}\right) \in L^1(\Rb^{d})$}.
\end{itemize}
\begin{lemma} \label{LemmagenStationSoln}
	{Let \textbf{(A0)} be true. If $\nu(x)$ is a smooth functions satisfying \textbf{(A0)}, then the unique stationary solution to the density given by the Fokker Planck equation \eqref{s2} is 
	\begin{equation}
		p^\infty(x) := \frac{1}{Z}\exp\left({-\frac{2}{\sigma^2}\bigg( w(x) \bigg)}\right)(x) \label{genCtrlSmolSoln},
	\end{equation}
	where $Z = \int  \exp\left({-\frac{2}{\sigma^2} w(x)}\right) \; \rd x$.}
\end{lemma}
\begin{proof}
	We observe that the \eqref{s2} is the Smoluchowski equation for an overdamped Langevin system given by:
	\begin{equation}
	\rd x_s = - \nabla (\nu+ v^\infty/R)(x_s) \; \rd s + \sigma \rd w_s \label{stationoptagentLangevin}.
	\end{equation}
	Under the assumptions above, the proof then follows directly from proposition 4.2, pp 110 in \cite{Pavliotis2014}.
\end{proof}

\section{Schr\"odinger Approach}
\label{sec:HopfCole}

The semilinear HJB PDEs above have a linear representation in the time-varying and steady state case. In the time-varying case it is obtained using the well-known Cole-Hopf 
transform $\phi := \exp(-v/\sigma^2 R)$, 
used in stochastic control theory by Kappen \cite{Kappen2005a}:
\begin{equation}
-\partial_t \phi = -\frac{q \phi}{\sigma^2 R} - \nabla \phi \cdot \nabla \nu + \frac{\sigma^2}{2} \Delta \phi. \label{HopfColeHJB}
\end{equation}
The advection-diffusion equation above has a path-integral solution \cite{Fleming1978} which is useful in computing the control \cite{Kappen2005a, Kappen2005b, Theodorou2011IPI}. In this section we apply two transforms providing a diffusion (no advection) Schr\"odinger PDE representation of the semilinear HJB \textit{and} linear FP equations. This facilitates a stability analysis of the steady state solution to the equations (\ref{HJB}, \ref{FP}) in section \ref{sec:Stab}, based on the spectral properties of the Schr\"odinger operator. Further, in this section we interpret the corresponding Schr\"odinger potential as the cost function of a fictitious but intimately related OCP with integrator dynamics. This motivates a quadrature based algorithm to solve the transformed HJB equation and compute the control in section \ref{sec:solver}.

\subsection{Application of the Cole-Hopf transform}
\label{sec:newHopfCole}

\indent \indent We apply the Cole-Hopf type transform:
\begin{equation}
f(t,x) := \exp(-(v(t,x) + R \nu(x))/\sigma^2 R), \label{newColeHopf}
\end{equation}
which leads to the following representation of equation \eqref{HJB}:
\begin{equation}
-\partial_t f = \frac{c f}{\sigma^2 R} -\frac{V f}{\sigma^2 R} + \frac{\sigma^2}{2} \Delta f = \frac{c f}{\sigma^2 R} - H f, \label{modHopfColeHJB}
\end{equation}
where we denote the modified cost function $V := q + (R/2)(\nabla \nu)^2 - (\sigma^2R/2) \Delta \nu$ and the operator $H := \frac{V}{\sigma^2 R} - \frac{\sigma^2}{2} \Delta$ is a Schr\"odinger operator with potential $\frac{V(x)}{\sigma^2 R}$. The transformed PDE can be verified by using the calculations $\partial_t v = -\sigma^2 R \frac{\partial_t f}{f} $, $\nabla f = -\frac{f}{\sigma^2 R} \nabla (v + R \nu)$, $\Delta f = -\frac{\nabla f}{\sigma^2 R} \cdot \nabla (v + R \nu) - \frac{f}{\sigma^2 R} \Delta (v + R\nu)$ and $\frac{(\nabla v)^2}{2R} = \big( \frac{\sigma^4 R}{2}\big( \frac{\nabla f}{f} \big)^2 + \sigma^2 R \frac{\nabla f}{f} \cdot \nabla \nu + \frac{R}{2}((\nabla \nu)^2 \big)$ in equation \eqref{HJB} thus recovering equation \eqref{modHopfColeHJB}. Similarly, it can be shown that if $v(t,x)$ is a solution of equation \eqref{HJB} then $f(t,x)$ given by \eqref{newColeHopf} is a solution to equation \eqref{modHopfColeHJB}.

\textit{Hermitizing} \cite{Ullmo2016} the density as:
\begin{equation}
	g := \frac{p}{f}, \label{hermit}
\end{equation}
then gives the following representation of equation \eqref{FP}:
\begin{equation}
-\partial_t g = -\frac{c g}{\sigma^2 R} + \frac{V g}{\sigma^2 R} - \frac{\sigma^2}{2} \Delta g = -\frac{c g}{\sigma^2 R} + H g, \label{modHopfColeFP}
\end{equation}
with the initial time boundary condition $g(0,x) = \frac{p}{f}(0,x)$ and normalizing constraint $\int f(t,x) g(t,x) \rd x = 1$ for all $t \geq 0$. This can be verified by using the derivatives $\partial_t p = \partial_t g f + g \partial_t f$, $\nabla p =  f \nabla g + g \nabla f$, $\Delta p =  f \Delta g + 2 \nabla g \cdot \nabla f + g \Delta f$, $\nabla (\sigma^2 \ln f) p = \sigma^2  g f \frac{\nabla f}{f} = \sigma^2  g \nabla f$ and equation \eqref{modHopfColeHJB} in equation \eqref{FP}, thus recovering the equation above. Similarly, it can be shown that if $p(t,x)$ is a solution of equation \eqref{FP} then $g(t,x) = \frac{p}{f}$, with $f(t,x)$ given by \eqref{newColeHopf}, is a solution to equation \eqref{modHopfColeFP}. We summarize this fact in the following theorem.
\begin{theorem} \label{ThmColeHopf}
	$( f(t,x), g(t,x) )$ is a solution to the linear PDEs (\ref{modHopfColeHJB}, \ref{modHopfColeFP}) such that $\int f(t,x) g(t,x) \rd x = 1$ for all $t \geq 0$ if and only if
	\begin{align}
	v(t,x) =& -\sigma^2R \ln (f)(t,x) - R \nu(x) \label{ColeHopfvinv} \\
	p(t,x) =& f(t,x) g(t,x) \label{ColeHopfpinv}
	\end{align}
	are solutions to the nonlinear equations \eqref{HJB} and \eqref{FP}. Further, the optimal control is given by $u^* = - \nabla v/R = \sigma^2 \nabla f/f$.
\end{theorem}

The applied transforms (\eqref{newColeHopf},\eqref{hermit}) correspond to a \textit{diagonalization} or decoupled, linear representation of the coupled equations (\eqref{HJB}, \eqref{FP}) as follows:
\begin{equation}
\partial_t \begin{bmatrix}
f \\ g
\end{bmatrix}
=
\begin{bmatrix}
H - \frac{c}{\sigma^2 R} \qquad \qquad 0 \\
0 \qquad \qquad -H + \frac{c}{\sigma^2 R}
\end{bmatrix}
\begin{bmatrix}
f \\ g
\end{bmatrix} \label{diag}
\end{equation}

Analogously, it can be shown that the stationary value and density functions satisfying the stationary nonlinear equations (\ref{s1}, \ref{s2})  can be represented by the transformation variables $f^\infty := \exp(-(v^\infty + R \nu)/\sigma^2 R)$ and $g^\infty := p^\infty/f^\infty$, which both satisfy the following eigenvalue problem:
\begin{align}
H e(x) =& \frac{c}{\sigma^2 R} e(x) \label{transs1s2}
\end{align}
subject to the normalizing constraint $\int f^\infty(x) g^\infty(x) \rd x = 1$.
\begin{theorem} \label{ThmColeHopfstatn}
	$\left(f^\infty(x), g^\infty(x) \right)$ are both solutions to the eigenvalue problem \eqref{transs1s2} such that $\int f^\infty(x) g^\infty(x) \rd x = 1$ if and only if
	\begin{align}
	v^\infty(x) =& -\sigma^2R \ln (f^\infty)(x) - R \nu(x) \label{ColeHopfstatnvinv} \\
	p^\infty(x) =& f^\infty(x) g^\infty(x) \label{ColeHopfstatnpinv}
	\end{align}
	are solutions to the nonlinear equations \eqref{s1} and \eqref{s2}). Further, the optimal control is given by $u^\infty = - \nabla v^\infty/R = \sigma^2 \nabla f^\infty/f^\infty$.
\end{theorem}
Given a solution pair $(v^\infty,p^\infty)$ to the equations \eqref{s1} and \eqref{s2}) it is possible to obtain explicit solutions to functions $(f^\infty,g^\infty)$ satisfying equation \eqref{transs1s2} such that $\int f^\infty g^\infty \rd x = 1$. The result in theorem \ref{LemmagenStationSoln} and the applied transforms can be used to verify the following corollary to theorem \ref{ThmColeHopfstatn}. 
\begin{corollary} \label{Corollaryfgstatnsol}
	Let $p^\infty : = \frac{1}{Z}\exp\left({-\frac{2}{\sigma^2}\bigg( w(x) \bigg)}\right)(x)$ with $w(x) := \nu(x) + \frac{v^\infty(x)}{R}$ and $Z$ the normalizing constant where $(v^\infty,p^\infty)$ is a pair satisfying \textbf{(A0)}. Then $f^\infty := \sqrt{Z  p^\infty}$ and $g^\infty := f^\infty/Z$ both satisfy equation \eqref{transs1s2} such that $\int f^\infty g^\infty \rd x = 1$.
\end{corollary}

\subsection{Interpretation} \label{sec:interpretHopfCole}
\label{sec:HopfCole_interpret}

\indent \indent The Schr\"odinger potential $\frac{V(x)}{\sigma^2 R}$ defined earlier can be interpreted as the cost function of the following \textit{fictitious} OCP which has an intimate connection with the original OCP \textbf{(P2)}:
\begin{equation}
\underset{\hat{u} \in \calU}{\min} J(u) := \underset{T \rightarrow +\infty}{\lim} \frac{1}{T} \Eb \left[ \int_{0}^{T} V(\hat{x}_s) \rd s + \frac{R}{2}\hat{u}^2\; \rd s \right] \label{modOCP}
\end{equation}
subject to the simple integrator dynamics:
\begin{equation}
\rd \hat{x}_s = \hat{u}(s) \rd s + \sigma \rd w_s. \label{lindyn}
\end{equation}
We refer to the OCP \eqref{modOCP} subject to \eqref{lindyn} as problem \textbf{(P2)}. The time-varying PDE representation of problem \textbf{(P2)} is given by:
\begin{align}
-\partial_t \hat{v} =&  V - \hat{c} - \frac{(\nabla \hat{v})^2}{2R} + \frac{\sigma^2}{2} \Delta \hat{v} \label{modHJB} \\
\partial_t \hat{p} =& \nabla \cdot (\frac{\nabla \hat{v}}{R}p) + \frac{\sigma^2}{2} \Delta \hat{p} \label{modFP}
\end{align}
where $\hat{c}$ is the optimal cost.

It is easily observed that if $v$ is the solution to the HJB equation \eqref{HJB}, then $\hat{v} = v + R \nu$ is a solution to the HJB equation \eqref{modHJB}. Therefore, the time-varying optimal controls: $u^*$ of the OCP \textbf{(P1)} and $\hat{u}^*$ of the OCP \textbf{(P2)}, are related as $\hat{u}^* = u^* - \nabla \nu$. Similarly, by substituting $\nabla \hat{v} = \nabla v + R \nabla \nu$ into equation \eqref{modFP}, we can see that the PDEs \eqref{FP}, \eqref{modFP} satisfied by the densities $p(s,x), \hat{p}(s,x)$ respectively, are identical. Therefore, given identical initial conditions $\hat{p}(0,x) = p(0,x)$, lemma \ref{LemmagenStationSoln} implies that $\hat{p}(s,x) = p(s,x)$ for all $s \geq 0$ where $p(s,x)$ is the density of optimally controlled agents associated with the OCP \textbf{(P1)}. To summarize, solving the equations (\ref{HJB}, \ref{FP}) corresponding to the OCP \textbf{(P1)} (subject to nonlinear passive dynamics) is equivalent to solving the equations (\ref{modHJB}, \ref{modFP}) corresponding to the OCP \textbf{(P2)} (subject to simple integrator dynamics). This fact is used in section \ref{sec:solver} to synthesize a solver to compute the finite-time optimal control.

\section{Control Design}
\label{sec:Stab}

The decay of an initial density of particles under open loop (uncontrolled) overdamped Langevin dynamics to a stationary density is a classic topic in statistical physics \cite{Risken1986}. In this section we analyze the decay of a perturbed density of agents under the action of the steady state controller obtained in problem \textbf{(P1)} to the corresponding stationary density. Since the HJB and FP equations \eqref{HJB} and \eqref{FP}) are coupled one-way, the perturbation analysis corresponds to that of the FP equation. In the presented approach, evolution of a perturbed density governed by equation \eqref{FP} is analyzed through the evolution of its hermitized form \eqref{hermit} governed by equation \eqref{modHopfColeFP}. Diagonalization \eqref{diag} of the PDE representation facilitates stability analysis based on the spectral properties of the Schr\"odinger operator, using which we obtain {analytical design constraints} on the {cost function $q(x)$} and {control parameter $R$} to guarantee stability of the equilibrium density. Note that our approach does not require the assumption of small (local) density perturbations or a Gaussian equilibrium as in \cite{Ullmo2016} which treated the stability of equilibrium densities in related MFG case, and applies to general perturbations to Gaussian or non-Gaussian stationary densities in the case of standard OCPs. Additionally, while the study in \cite{Ullmo2016} is limited to agents which have simple integrator dynamics the proposed framework applies to the broader class of nonlinear Langevin dynamics.

\subsection{Perturbations of the equilibrium density}

\indent \indent Consider a large-size population expressed by problem \textbf{(P1)}, which is controlled by the optimal steady state control $u^\infty = -\nabla v^\infty/R$ corresponding to the equation \eqref{s1} with a unique equilibrium density $p^\infty$ satisfying \eqref{s2} and \textbf{(A0)}. Theorem \ref{ThmColeHopfstatn} then implies that the value and density functions can be written as \eqref{ColeHopfstatnvinv}, \eqref{ColeHopfstatnpinv}, in terms of a pair of functions $(f^\infty,g^\infty)$, both satisfying equation \eqref{transs1s2} and $\int f^\infty g^\infty \rd x = 1$. Corollary \ref{Corollaryfgstatnsol} gives formulae for the function pair $\left(f^\infty(x), g^\infty(x) \right)$ in terms of the steady state solution pair $(v^\infty,p^\infty)$ satisfying equations (\ref{s1}, \ref{s2}). Time-varying value and density functions can be calculated as \eqref{ColeHopfvinv}, \eqref{ColeHopfpinv} in terms of the corresponding transformation variables $(f(t,x),g(t,x))$.

Time-varying densities, perturbed from the steady state density of agents can be written using the hermitization transform \eqref{hermit} as $p(t,x) = p^\infty(x) + \tilde{p}(t,x) = f^\infty(x) g^\infty(t,x) + f^\infty(x) \tilde{g}(t,x)$. Since we are studying stability of the steady state controller, there are no perturbations in the value function $v^\infty$ nor consequently, in the transformation variable $f^\infty$. Here, the function $\tilde{g}(t,x)$ corresponds to a perturbation in the hermitized density given as $g(t,x) = g^\infty(x) + \tilde{g}(t,x)$, which obeys the time-varying PDE \eqref{modHopfColeFP}. In this section we study the decay of a perturbed density $p^\infty + \tilde{p}$ to its steady state density $p^\infty$. We state the following corollary to theorem \ref{ThmColeHopf} which provides the perturbation equation for the hermitized density.
\begin{corollary}
	If $g^\infty(x)$ is a solution to the stationary PDE \eqref{transs1s2} and $g(t,x) = g^\infty(x) + \tilde{g}(t,x)$ is a solution to the PDE \eqref{modHopfColeFP} where $\tilde{g}(t,x) \in C^{1,2}([0,+\infty), \Rb^d)$, then $\tilde{g}(t,x)$ is governed by the linear PDE 
\begin{equation}
\partial_t \tilde{g} = -\left(H - \frac{c}{\sigma^2 R}\right) \tilde{g}. \label{transp2}
\end{equation}
\end{corollary}

\subsection{Stability}
\label{sec:Stab.stab}

\indent \indent We define the following Hilbert space and class of density perturbations for which we study stability.

\begin{definition}
Let \textbf{(A0)} hold. Denote $p^\infty : = \frac{1}{Z}\exp\left({-\frac{2}{\sigma^2}\bigg( w(x) \bigg)}\right)(x)$ with $w(x) := \nu(x) + \frac{v^\infty(x)}{R}$ and $Z$ the normalizing constant where $(v^\infty,p^\infty)$ is the unique pair satisfying \textbf{(A0)}. We denote by $f^\infty := \sqrt{Z p^\infty}$ and $g^\infty := f^\infty/Z$ two solutions to equation \eqref{transs1s2} such that $\int f^\infty g^\infty \rd x = 1$. We denote the Hilbert space of $L^2(\Rb)$ by $\calH$. The class of mass preserving density perturbations is defined as $S_0 := \left\{ \pi(x)\in \calH \bigg| \left< \pi, f^\infty \right>_\calH = 0 \right\}$.
\end{definition}

\begin{definition}
We define the class of initial perturbed densities as $S :=$ \\
$\left\{ p(0,x) = f^\infty (g^\infty(x) + \tilde{g}(0,x)) \bigg| \tilde{g}(0,x) \in S_0 \right\}$.
We say that the steady state density $p^\infty(x) = f^\infty(x) g^\infty(x)$ corresponding to the nonlinear equations (\ref{s1}, \ref{s2}) is asymptotically stable with respect to $S$ if there exists a solution $\tilde{g}(t,x)$ to the perturbation equation \eqref{transp2} such that $\underset{t \rightarrow +\infty}{\lim}||\tilde{g}(t,x)||_\calH = 0$.
\end{definition}

\begin{lemma}\label{LemmaSelfAdj}
If there exists a positive, even and continuous function $Q(x)$ on $\Rb$ which is non-decreasing for all $x \geq 0$ such that $\frac{V(x)}{\sigma^2 R} \geq -Q(x)$ for all $x \in \Rb$ and $\int \frac{\rd x}{\sqrt{Q(2x)}} \rd x = +\infty$ then the closure of $H$ is self adjoint.
\end{lemma}

We omit the proof since it follows directly from theorem 1.1, pp 50 in \cite{Berezin1991}. In particular, if $\frac{V(x)}{\sigma^2 R} \geq \text{k} \in \Rb$ then it follows that $H$ is self adjoint. The following assumption implies discreteness of the spectrum of $H$.

\begin{itemize}
\item[\textbf{(A1)}] $\underset{|x| \rightarrow +\infty}{\lim} V(x) = \underset{|x| \rightarrow +\infty}{\lim} q + \frac{R}{2}(\nabla \nu)^2 - \frac{\sigma^2R}{2} \Delta \nu =  +\infty$.
\end{itemize}

\begin{lemma}\label{LemmaSpecGap}
If \textbf{(A1)} is true then the closure of $H$ has a discrete spectrum.
\end{lemma}
The proof of this theorem follows from theorem 3.1, pp 57 of \cite{Berezin1991}. This theorem implies that under assumption \textbf{(A1)}, the spectrum of $H$ denoted by $\{\lambda_n\}_{0 \leq n \leq +\infty}$ has the property that $\lambda_n \rightarrow +\infty$ as $n \rightarrow +\infty$ and the corresponding eigenfunctions denoted as $\{e_n(x)\}_{0 \leq n \leq +\infty}$ form a complete orthonormal system on $L^2(\Rb)$. The eigenproperty is explicitly written as $H e_n(x) = \lambda_n e_n(x)$. Further from proposition 3.2, pp 65 in \cite{Berezin1991} the eigenvalues have the property $\lambda_0 < \lambda_1 < \cdots < \lambda_n < \cdots$. We state the following assumption on the Schr\"odinger potential required to prove stability of the stationary density.
\begin{itemize}
\item[\textbf{(A2)}] $V(x) = q + \frac{R}{2}(\nabla \nu)^2 - \frac{\sigma^2R}{2} \Delta \nu \geq 0$.
\end{itemize}
In view of \textbf{(A1)}, \textbf{(A2)} can be satisfied for given $q$ and $\nu$ functions without affecting the equilibrium solutions $(v^\infty, p^\infty)$, by adding a constant to the cost function. This condition is required to apply the Krein-Rutman theorem in the proof of the main theorem below as in (section 1.A, pp 2, \cite{todorov2009eigenfunction}). The proof of the theorem uses standard techniques from spectral analysis of linear operators and is presented here for the purposes of completeness and explaining the precise role of the assumptions \textbf{(A1, A2)} in providing the guarantee of stability. We do not lay any claims of originality to the techniques of the proof.

\begin{theorem}\label{ThmStab}
Let \textbf{(A0, A1, A2)} be true. Let $\left(v^\infty(x), p^\infty(x) \right)$ be the unique stationary solution to the equations (\ref{s1}, \ref{s2}) and denote by $(f^\infty,g^\infty)$ the two solutions to problem \eqref{transs1s2} given in corollary \ref{Corollaryfgstatnsol}. If $\tilde{g}(0,x) \in S_0$ and $\{g_n\}_{0 \leq n \leq +\infty}$ are determined by 
\begin{equation}
\dot{g_n}(t) = -\left(\lambda_n - \frac{c}{\sigma^2 R} \right) g_n(t), \label{pertcoeffsODE}
\end{equation}
then $\tilde{g}(t,x) = \sum_{n = 1}^{+\infty} g_n(t) e_n(t)$ is the unique $\calH$ solution to the perturbation equation \eqref{transp2}. $p^\infty(x)$ is asymptotically stable with respect to $S$.
\end{theorem}
\begin{proof}
Since $\tilde{g}(0,x) \in \calH$ we have the unique representation $\tilde{g}(0,x) = \sum_{n = 0}^{+\infty} g_n(0) e_n(x)$ where $g_n(0) = \left< \tilde{g}(0,x), e_n(x) \right>_\calH < +\infty$ for all $n$.
Since $\{e_n\}_{0 \leq n < +\infty}$ is a complete basis on $\calH$, any solution in $\calH$ to the PDE \eqref{transp2} must have the form $\sum_{n = 0}^{+\infty} g_n(t) e_n(x)$ where $\{g_n(t)\}_{0 \leq n \leq +\infty}$ are finite for all $t \in [0,+\infty)$. Substituting the selected form of the solution in the perturbation equation \eqref{transp2} and using the eigenproperty $H e_n = \lambda_n e_n$, we obtain the ODEs \eqref{pertcoeffsODE}. Due to assumption \textbf{(A1)} the eigenproperties of the Schr\"odinger operator given in lemmas \ref{LemmaSelfAdj}, \ref{LemmaSpecGap} hold. Using the eigenproperty yields the ODEs \eqref{pertcoeffsODE} with the unique solutions ${g_n}(t) = g_n(0) \re^{-\left(\lambda_n - \frac{c}{\sigma^2 R}\right) t}$. Therefore $\tilde{g}(t,x) = \sum_{n = 0}^{+\infty} g_n(t) e_n(x)$ wherein ${g_n}(t) = g_n(0) \re^{-\left(\lambda_n - \frac{c}{\sigma^2 R}\right) t}$ is the unique $\calH$ solution to the perturbation equation \eqref{transp2}.

{From the Krein-Rutman theorem \cite{todorov2009eigenfunction}, under the assumption $V(x) \geq 0$ given by \textbf{(A2)}, the first eigenvalue is $\frac{c}{\sigma^2 R} = \lambda_0$ so that $\dot{g_0}(t) = 0$ for all $t \geq 0$ and the first eigenfunction is $0 < f^\infty(x) = e_0(x)$ corresponding to the eigenproblem \eqref{transs1s2}. Further, $\tilde{g}(0,x) \in S_0$ implies that $g_0(0) = \left< \tilde{g}(0,x), e_0(x) \right>_\calH = \left< \tilde{g}(0,x), f^\infty(x) \right>_\calH = 0$ implying that $g_0(t) = 0$ for all $t \geq 0$}, which completes the first part of the proof.

We can see that $\lambda_0 > 0$ since the infinite horizon optimal cost $c> 0$ or using integration by parts, $\left< H e_0, e_0 \right>_{L^2(\Rb)} = \left< \frac{V}{\sigma^2 R} e_0, e_0 \right>_{L^2(\Rb)} + \frac{\sigma^2}{2} ||\nabla e_0||^2_{L^2(\Rb)} = \lambda_0 = \frac{c}{\sigma^2 R} > 0$ since $V(x) \geq 0$. From assumption \textbf{(A2)} and $\lambda_0 < \lambda_1 < \cdots$ due to assumption \textbf{(A1)} wherein $\frac{c}{\sigma^2 R} = \lambda_0$, we conclude that $\lambda_n - \frac{c}{\sigma^2 R} > 0$ for all $n > 1$. Using Parseval's identity $||\tilde{g}(t,x)||_{L^2(\Rb)} = \left( \sum_{n = 0}^{+\infty} g_n(t)^2 \right)^\frac{1}{2}$, noting that $g_0(t) = 0$, $g_n(t)^2 = g_n(0)^2 \re^{-2\left(\lambda_n - \frac{c}{\sigma^2 R} \right) t}$ where $\lambda_n - \frac{c}{\sigma^2 R} > 0$ for all $n > 1$ and using the Lebesgue dominated convergence theorem for the limit $t \rightarrow +\infty$, we have that $p^\infty(x)$ is nonlinearly asymptotically stable with respect to $S$.
\end{proof}
From the theorem above, we note that assumption \textbf{(A1)} is the key explicit design constraint on the cost function $q(x)$ and control parameter $R$, which guarantees stability of an initially perturbed density of agents to the corresponding steady state density, under the action of the steady state controller. In figure \ref{fig:Stab} we show stabilization of an initially (perturbed) uniform density of agents to the stationary density corresponding to the steady state controls. The agent dynamics are unstable with the Langevin potential $\nu(x) = -x^3/3$ and the system is stabilized using a cost function $q(x) = (5/2) \cdot x^2$ such that conditions \textbf{(A1, A2)} are satisfied. Equation \eqref{transs1s2} is solved using a spectral solver \cite{driscoll2014chebfun} for the parameters $\sigma = R = 1/2$ and the steady state density is obtained using equation \eqref{stationoptagentLangevin}. Initial states of agents are sampled from a uniform density over the interval $[-2,2]$. Trajectories for $N = 500$ agents are simulated with 100 stochastic realizations each, using the steady state control. In the left panel we observe the density evolve over time steps $t = 0$ (black), $t = T/5$ (blue), $t = T/2$ (pink) to the final time $t = T$ (red) at which the stationary density computed by the spectral solver is recovered.   
\begin{figure}
\centering
	\begin{minipage}{.45\textwidth}
		\centering
		\includegraphics[width=1\linewidth]{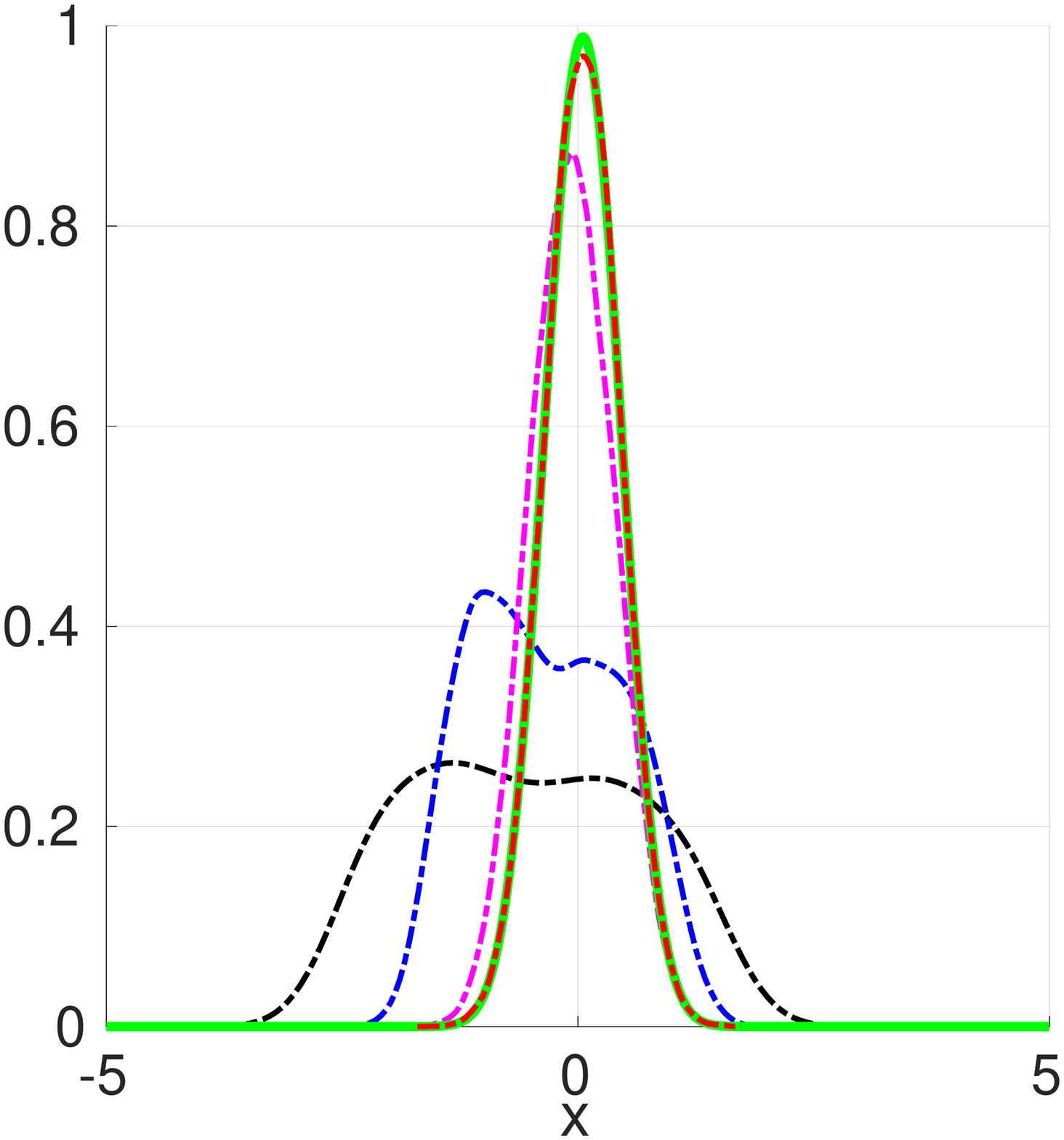}
		\label{fig:test1}
	\end{minipage}
	\begin{minipage}{.45\textwidth}
		\centering
		\includegraphics[width=1\linewidth]{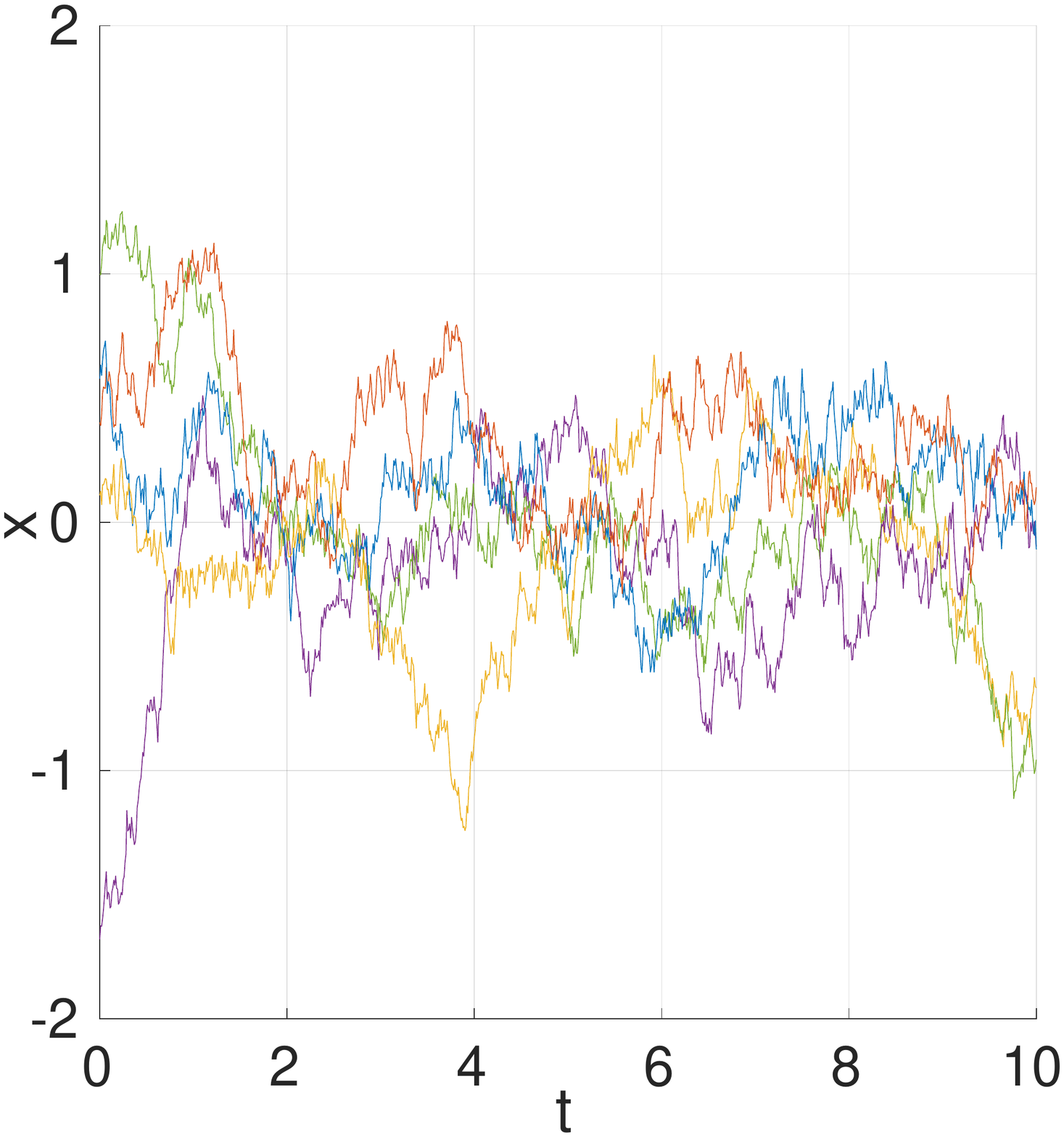}
		\label{fig:test2}
	\end{minipage}
	\caption{Stabilization of a density of unstable Langevin agents with potential $\nu = -x^3/3$  to a equilibrium density by designing an OCP which satisfies the stability constraints \textbf{(A1,A2)}. \textit{Left pane:} Monte-Carlo simulation of density evolution under steady state controls from the perturbed initial density (black, $t = 0$) - blue ($t=T/5$) - pink ($t = T/2$) - terminal (red, $t = T$) density. Note that the terminal (red) density obtained via Monte-Carlo co-incides with the equilibrium density computed by the spectral solver (green). \textit{Right pane:} Stochastic realizations of ten optimally controlled agent trajectories.}
    \label{fig:Stab}
\end{figure}

\section{Control Algorithm}
\label{sec:solver}
For practical applications of the control of large-scale systems, it will be advantageous to pre-compute a finite-time, feedback control law whose domain spans the region of state space that we are interested in. The optimal control can be obtained by solving the corresponding HJB PDE by using finite difference, finite element, spectral or path-integral approaches. In this work, we apply a path-integral approach to solve this PDE in the finite-time case and introduce an efficient quadrature method for evaluating the path-integrals. Although our quadrature method could be applied to either \textbf{(P1)} or \textbf{(P2)}, the implementation becomes computationally efficient in the case of \textbf{(P2)} due to our analytic knowledge of the density for the underlying integrator sampling dynamics, in contrast to to prior path-integral sampling control algorithms (\cite{Kappen2006}, \cite{Theodorou2011IPI}) which require simulation of trajectories of nonlinear dynamics  to numerically approximate the probability density. The result is an efficient method for computing the feedback control law.  

We consider the finite horizon OCP \eqref{finitetimeOCP} with the HJB equation given by \eqref{HJB}, $c = 0$ as explained in remark \ref{finitetime}. The optimal control can then be solved by treating the equivalent problem \textbf{(P2)} with HJB equation \eqref{modHopfColeHJB}. The path-integral representation of this PDE (via Feynman-Kac) is as follows:
\begin{equation}
f(t,\bar{x})=\Eb_\tau\Bigg[\exp\Big(\int_t^T-\frac{V}{\sigma^2R}(x_s)ds\Big)f(T,x_T)\Bigg]
\end{equation}
with the expectation over trajectories $\tau$ of brownian motions over the finite-time horizon $[t,T]$, that is:
\begin{equation}
d x_s = \sigma d\omega_s, \quad x_t=\bar{x}
\end{equation}
First we approximate everything in discrete time with $N$ timesteps of duration $\delta t$, with $\delta t=(T-t)/N$, so that
\begin{equation}
f(t,\bar{x})\approx\Eb_\tau\Bigg[\exp\Big(\sum_{n=0}^{N-1}-\frac{V}{\sigma^2R}(x_n)\delta t\Big)f(T,x_N)\Bigg] \label{discreqn}
\end{equation}
with $x_n$ governed by the discrete dynamical system $
x_{n+1} = x_n + \sigma \sqrt{\delta t}\epsilon$, $ x_0=\bar{x}$, and $\epsilon\sim\vcalN(0,I)$,
with the associated transition probability $p(x_{n+1}|x_n)\sim\vcalN(x_n,\sigma^2\delta t \mathbf{I})$. Let $w_n(x_n) := \exp\Big(-\frac{V}{\sigma^2R}(x_n)\delta t\Big)$ for $n=0,\cdots,N-1$, and let $w_N(x_N) := f(T,x_N)$, and $w := \prod_{n=0}^{N} w_n(x_n)$.

From equation \eqref{discreqn} we have that $f(t,\bar{x})=$
\begin{align}
\begin{split}
&\int\cdots\int w_N(x_N)\Bigg[\prod_{n=2}^{N-1} w_n(x_n)p(x_{n+1}|x_n) \Bigg] \times \\
& \Bigg[\int w_0(\bar{x})p(x_1|x_0=\bar{x})w_1(x_1)p(x_2|x_1)dx_1\Bigg]dx_2\cdots dx_N.\hspace{-10pt}
\label{eq:quad0}
\end{split}
\end{align}

The second integral in the brackets above is approximated by Gaussian quadrature with M grid points $\{\xi_1^i\}_{i=1}^M$ and weights $\alpha_1^i$ as:
\begin{equation}
\sum_{i=1}^M \underbrace{p(x_2|x_1=\xi_1^i)}_{\phi_1^i(x_2)}\underbrace{\alpha_1^i w_1(x_1=\xi_1^i)}_{\gamma_1^i}\underbrace{w_0(\bar{x})p(x_1=\xi_1^i|x_0=\bar{x})}_{\phi_0^i(\bar{x})}.
\label{eq:quad1}
\end{equation}

\noindent Define the $M$ dimensional vectors $\Phi_1(x_2)$, $\gamma_1$, and $\Phi_0(\bar{x})$ to have elements $\phi_1^i(x_2)$, $\gamma^i_1$, $\phi^i_0(\bar{x})$, respectively and define $\Gamma_1=diag(\gamma_1)$, and write (\ref{eq:quad1}) as a set of vector products $\int w_0(\bar{x})p(x_1|x_0=\bar{x})w_1(x_1)p(x_2|x_1)dx_1 = \Phi_1(x_2)^\intercal \Gamma_1\Phi_0(\bar{x})$.
Recall that $p(x_1=\xi_1^i|x_0=\bar{x})$ is a Gaussian density, so that each element of $\Phi_0(\bar{x})$ is Gaussian weighted by $w_0(\bar{x})$.
Plugging this back into (\ref{eq:quad0}) yields:
\begin{align}
\begin{split}
&=\int\cdots\int w_N(x_N)\Bigg[\prod_{n=3}^{N-1} w_n(x_n)p(x_{n+1}|x_n) \Bigg]\\
&\qquad \Bigg[\int w_2(x_2)p(x_3|x_2)\Phi_1(x_2)^\intercal \Gamma_1\Phi_0(\bar{x})dx_2\Bigg] dx_3 \cdots dx_N.\hspace{-20pt}
\label{eq:int1}
\end{split}
\end{align}
Performing another quadrature on the integral within the brackets at points $\{\xi_2^i\}_{i=1}^M$ with weights $\alpha_2^i$ we obtain:
\begin{equation}
\sum_{i=1}^M \underbrace{p(x_3|x_2=\xi_2^i)}_{\phi^i(x_3)}\underbrace{\alpha_2^i w_2(x_2=\xi_2^i)}_{\gamma_2^i}\Phi_1(x_2=\xi_2^i)^\intercal \Gamma_1\Phi_0(\bar{x})
\label{eq:quad2}
\end{equation}
Let $\tilde{\bm{\Phi}}_n$ be an $M\times M$ transition matrix with elements $\{\tilde{\bm{\Phi}}\}_{ij}=p(x_{n+1}=\xi_{n+1}^i|x_n=\xi_n^j)$. Then we can write (\ref{eq:quad2}) as
$\Phi_2(x_3)^\intercal \Gamma_2\tilde{\bm{\Phi}}_1\Gamma_1\Phi_0(\bar{x})$.
Plugging this back into (\ref{eq:int1}), we can perform the nested integrals recursively. At each timestep $x_n$ we use a different quadrature grid, with points $\{\xi_n^i\}_{i=1}^M$ and weights $\alpha_n^i$. The entire integral will therefore be:
\begin{equation}
	f(t,\bar{x})\approx \gamma_N^\intercal\Bigg[\prod_{n=1}^{N-1}(\tilde{\bm{\Phi}}_n\Gamma_n)\Bigg]\Phi_0(\bar{x})
\label{eq:final1}
\end{equation}
where we have used the definitions:
\begin{align}
\gamma_n&=\Big[\{\alpha_iw_n(\xi_n^i)\}_{i=1}^M\Big]^\intercal\\
\Gamma_n&=diag(\gamma_n)\\
\{\tilde{\bm{\Phi}}_n\}_{ij}&=p(x_{n+1}=\xi_{n+1}^i|x_n=\xi_n^j)\\
\phi_0^i(\bar{x})&= w_0(\bar{x})p(x_1=\xi_1^i|x_0=\bar{x})\\
\Phi_0(\bar{x})&=\Big[\{\phi_0^i(\bar{x})\}_{i=1}^M\Big]^\intercal
\end{align}

Since $V(x)$ is time invariant and one chooses the same quadrature grid points at each timestep, $\gamma_n$ and $\tilde{\bm{\Phi}}_n$ are the same for all $n=1,\cdots,N-1$.  So (\ref{eq:final1}) can be simplified to:
\begin{equation}
f(t,\bar{x})\approx \gamma_N^\intercal(\tilde{\bm{\Phi}}\Gamma)^{N-1}\Phi_0(\bar{x}).
\end{equation}

We consider a two dimensional Langevin potential $\nu = 1/2\cos(x_1 x_2)^2 - 1/24(x_1^4+x_2^4)$, with resulting dynamics:
\begin{align*}
\begin{split}
\rd x_1 &= (\cos(x_1 x_2)\sin(x_1 x_2)x_2-1/6x_1^3 + u_1(s))\rd s + \sigma \rd w_1\\
\rd x_2 &= (\cos(x_1 x_2)\sin(x_1 x_2)x_1-1/6x_2^3 + u_2(s))\rd s + \sigma \rd w_2.
\end{split}
\end{align*}

\noindent In Figure \ref{fig:nu} we first plot the potential $\nu$ along with several uncontrolled trajectories of agents initialized at random locations.  The agents collect into 4 stable and attracting equilibria.  We design a cost function $q(x) = \frac{1}{2}Q((x_1-1)^2+(x_2-1)^2)((x_1+1)^2+(x_2+1)^2)$ to encourage the agents to move into two locations at $(-1,-1)$ and $(1,1)$.  We let $R=1, Q=0.1, \sigma=0.2$, $T=4.0s$, and $dt=0.1$.  We solve for $f(t,x)$ at each timestep using our quadrature method with a fixed 2-d Gauss-Hermite grid spanning $[-2,2]$ in both $x_1$ and $x_2$.  We found 20 grid points in each dimension to yield good results.  We then plot the modified value $\hat{v}(x,t)=-\sigma^2/2\log(f(x,t))$.  With this method we are able to find an optimal feedback control law for the entire domain of integration (Figure \ref{fig:nu}, lower four).  Note that we are also able to solve the problem by calculating controls for each agent locally and independently using our quadrature method, modified to use a smaller grid (with width $4\sigma (T-t)/\sqrt{dt}$ in each dimension), centered at the agent's current position.  Unlike with PDE solver-based solutions, we are able to find optimal controls for each agent locally.  This is advantageous when the size of the state space is large and the number of agents is small. The results of the simulation show that early on $(t=1.0s)$, the agents are pushed towards the center of the space.  As time progresses, the agents are controlled towards the goal position at $(1,1)$ and $(-1,-1)$ for $(t=2.0s,3.0s)$.  At the final time, the agents are mainly concentrated around the goal regions $(t=4.0s)$.  The modified value $\hat{v}$ is smallest at the goal state but also has valleys around the four stable equilibria.  We make our code publicly available at \url{https://github.com/ddfan/pi_quadrature}.  On an Intel(R) Core(TM) i7-4980HQ CPU @ 2.80GHz machine, calculating the value function and simulating the agents took 21.7 seconds, using Python and Numpy's linear algebra library.

\begin{figure}
\centering
	\begin{subfigure}
		\centering
	\includegraphics[width=0.45\textwidth]{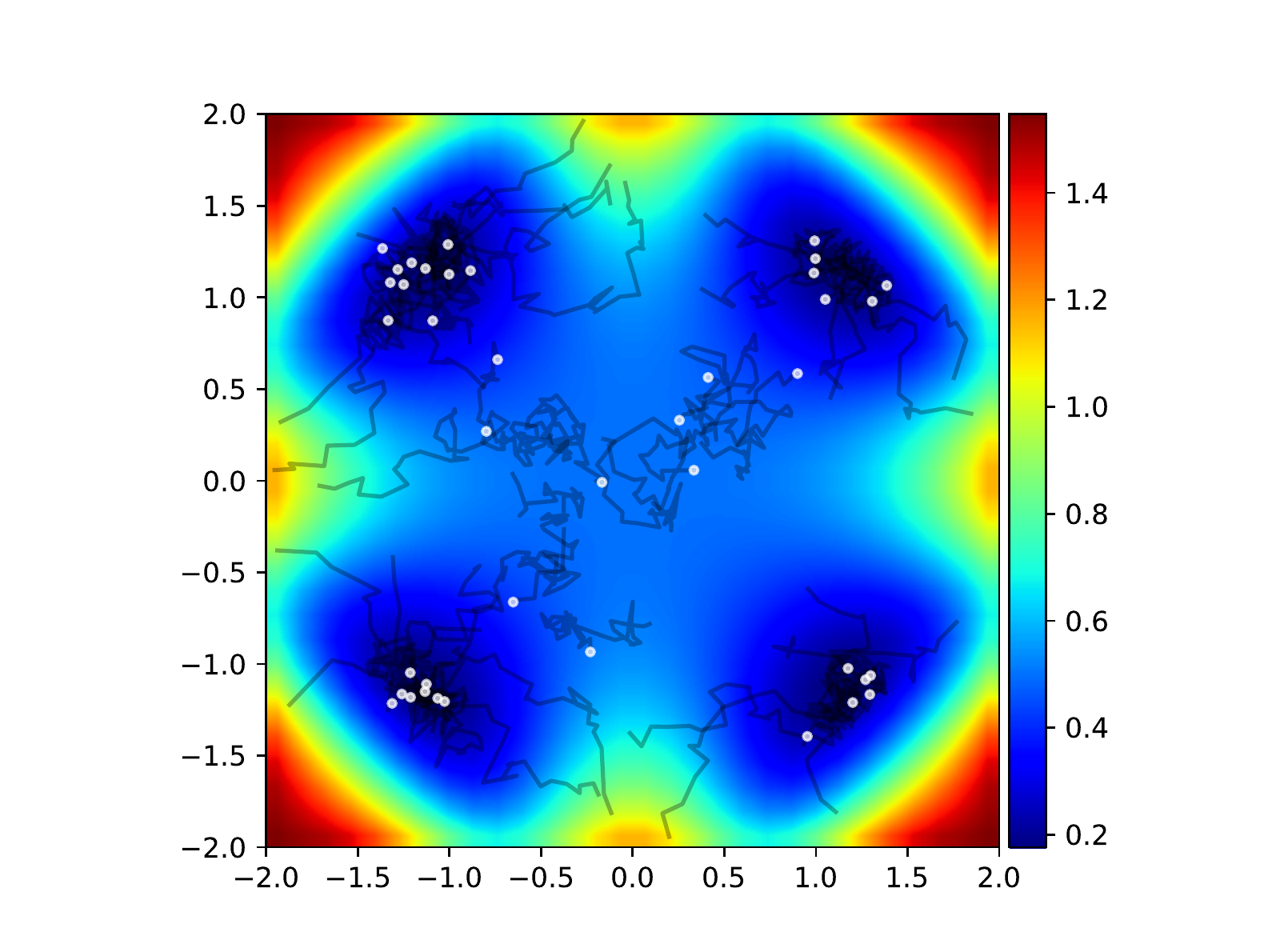}
  \end{subfigure}
  \begin{subfigure}
	\centering
	\includegraphics[width=0.55\textwidth]{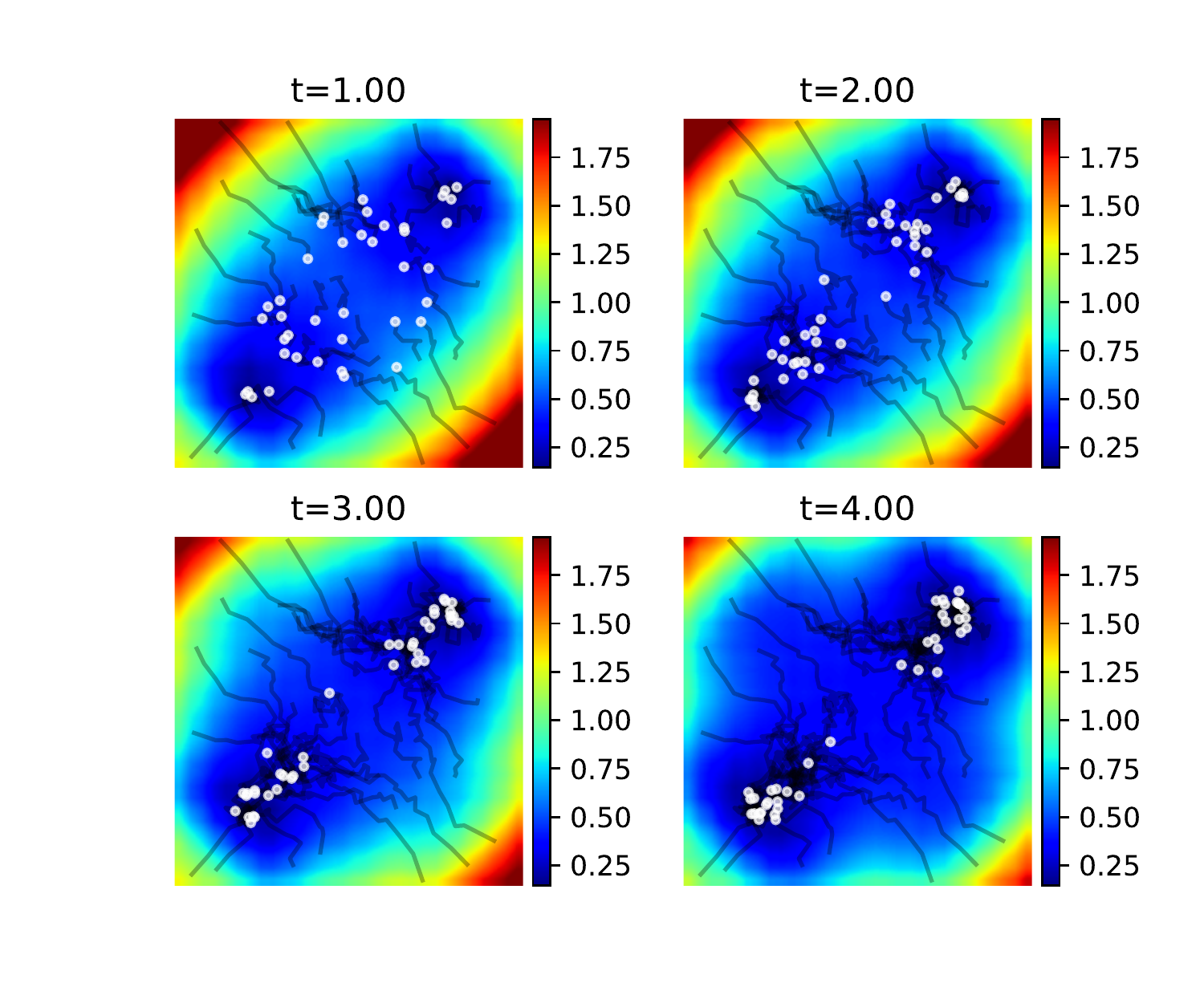}
    \end{subfigure}
	\caption{Top:  Langevin potential $\nu$ for 2 dimensional problem.  Trajectories of 40 agents under no control (black lines) along with their final position after $T=4.0$ seconds (white dots) are plotted.  Lower 4: Optimally controlled agents and value for 2 dimensional finite-horizon problem (T=4.0s).  Color denotes plot of $\hat{v}(x,t)=-\sigma^2/2\log(f(x,t))$.  4 snapshots in time are shown.  Note that the agents move towards the regions of lowest cost at $(-1,-1)$ and $(1,1)$ but are affected by the other potential wells at $(\pm 1,\pm 1)$.}
\label{fig:nu}
\end{figure}

\section{Conclusions}
\label{sec:conclusions}

We have presented an imaginary-time, linear and decoupled Schr\"odinger equation representation of the HJB and FP equations of a class of stochastic OCPs for agents with multi-dimensional \textit{nonlinear Langevin dynamics}. This representation is obtained by introducing novel variable transforms of the value \textit{and} density functions. Both the HJB and FP equations become Schr\"odinger equations with identical potential, in the transformed variables. 

We interpret the Schr\"odinger potential as the cost function of a related OCP with simple integrator dynamics. This motivates a quadrature based algorithm to compute the finite-time control and is demonstrated on a two dimensional example.

The Schr\"odinger representation of the HJB and FP PDEs facilitates a stability analysis of the density without having to relying on linearization of the FP equation at the equilibrium density. The proposed approach provides a framework for closed-loop \textit{nonlinear stability} analysis of the (in general) \textit{non-Gaussian} steady state density, using which, we obtain explicit, analytic \textit{control design} constraints to guarantee stability. It is observed that spectral properties of the Schr\"odinger operator associated with the HJB and FP equations underlie the stability of equilibrium density.

The connection between the nonlinear Schr\"odinger equation and MFGs was used in combination with the \textit{theory of solitons} \cite{Ullmo2016} to study reduced order, quaratic-Gaussian approximations of the solutions of MFGs for agents with simple integrator dynamics. A topic of future work is to extend and apply the Schr\"odinger representation introduced in this work, to the case of MFGs for agents with nonlinear Langevin dynamics. The theory of solitons can then be leveraged as in \cite{Ullmo2016} to create a reduced order computational tool for this broader class of MFGs, with the ultimate goal of designing phase transitions (\textit{operating regimes}) in multi-agent networked systems, such as agile swarms \cite{Bauso2018} and electrical micro-grids \cite{Bullo2013}.

Generalization of the presented approach to the case of second order Langevin systems is a natural extension, which is a work currently underway by the authors.

Finally, we will introduce sparse grids \cite{Narayan2018} in the proposed quadrature based finite-time optimal control solver in a forthcoming publication, with the goals of speeding up computation and scaling to high dimensional systems.

\bibliographystyle{unsrt}
\bibliography{References}

\end{document}